\pdfoutput=1

\documentclass[11pt,a4paper]{amsart}

\usepackage{amscd}
\usepackage{setspace}
\usepackage{a4wide}
\usepackage{amsmath, amsthm}
\usepackage{amssymb, amsfonts}
\usepackage{graphicx}
\usepackage{datetime}
\usepackage{epstopdf}

\newtheorem{theorem}{Theorem}[section]

\newtheorem{lemma}[theorem]{Lemma}

\newtheorem{assumption}[theorem]{Assumption}
\newtheorem{remark}[theorem]{Remark}

\newcommand{\Exp}{\operatornamewithlimits{Exp}}
\newcommand{\Geom}{\operatornamewithlimits{Geom}}

\numberwithin{equation}{section}

\title[Multiple Entries with Forced Exits]{A class of Solvable Multiple Entry Problems with Forced Exits}
\author{Jukka Lempa}
\address{Jukka Lempa, Department of Mathematics and Statistics, University of Turku, FI - 20014 Turun Yliopisto, Finland, \texttt{jumile@utu.fi}}
\date{\today, \currenttime}

\begin{document}

\begin{abstract}
We study an optimal investment problem with multiple entries and forced exits. A closed form solution of the optimisation problem is presented for general underlying diffusion dynamics and a general running payoff function in the case when forced exits occur on the jump times of a Poisson process. Furthermore, we allow the investment opportunity to be subject to the risk of a catastrophe that can occur at the jumps of the Poisson process. More precisely, we attach IID Bernoulli trials to the jump times and if the trial fails, no further re-entries are allowed. We show in the general case that the optimal investment threshold is independent of the success probability is the Bernoulli trials. The results are illustrated with explicit examples.   
\end{abstract}

\maketitle

\section{Introduction}

We consider an investment model in continuous time where the decision maker has the option to invest in a given project yielding uncertain returns $X$. The investors objective is to choose the entry time such that a particular objective functional (often of either discounted or ergodic type) is maximised. At the time of the entry, a known fixed lump sum $K$ of entry costs must be paid. Once the entry is made, the investment incurs a known constant instantaneous running cost $C$. In the classical perpetual version of this problem, see, e.g., \cite{Dixit, DP}, it is assumed that once the entry is made, the return from the investment will accrue from the investment date to infinity. A variant of this problem includes the possibility of voluntary exits, see, e.g., \cite{BZ, DZ, Zervos}. The problem becomes then of sequential nature, where a sequence of entry and exit times is determined such that the objective functional is maximised.

The purpose of this paper is to study a class of multiple entry problems where the exits are not voluntary but forced. This type of problem was first studied in \cite{Wang2005} and can be informally described as follows. Subject to return uncertainty modelled by a time homogeneous diffusion process $X$, and known entry and running costs $K$ and $C$, the investor chooses the time of entry. However, the investment is subject to exogenous interventions, which will terminate the flow of returns. These interventions occur uniformly over time and are modelled by the jumps of a Poisson process $N$ independent of $X$. Once the investment is terminated, the decision maker can make a new entry. The objective is then to maximise the expected present value of the total revenue from the investment. As was discussed already in \cite{Wang2005}, this setting lends itself to a possible application to a so-called liquidation risk \cite{DR}. Consider the case where the investor is funding the investment with borrowed money. To increase liquidity on the lenders side, it is possible that the lender is given (or requires) the right to seize the asset and put it to alternative use, that is, liquidate. Thus the intervention would be forced liquidation from the lenders side. After a forced liquidation, the investor can find a new lender to make a new entry.

This paper makes two contributions. First, we allow the underlying stochastic process $X$ to follow a general one-dimensional diffusion process with natural boundaries. In comparison to \cite{Wang2005}, where the case of geometric Brownian motion is considered, this is a substantial generalisation which has not, to our best knowledge, been studied earlier in the literature. Furthermore, we introduce a so-called catastrophe risk in the model as follows: we attach IID Bernoulli trials to the jump times of $N$ and if the trial fails, no further re-entries are allowed. In the liquidation risk application described above, the catastrophe event describes a fundamental change in the economical environment of the investment opportunity such that all lenders lose interest in financing a new entry. Such a change could be, for instance, due to the emergence of a new technology. This is again a substantial generalisation and it effectively means that the number of forced exits up to the catastrophe is geometrically distributed. Somewhat remarkably we find that the optimal investment threshold is independent of the success probability of the Bernoulli trials. 

The paper is organised as follows. In Section 2 we set up the probabilistic framework. In Section 3 the entry problem is introduced. The solution is derived in Section 4. The paper is concluded in Section 5 with illustrative examples.

\section{The Probabilistic Setting}

We set up the probabilistic framework for the entry problem, for details, see \cite{BoSa}. Let $(\Omega,\mathcal{F},\mathbb{F},\mathbf{P})$, where $\mathbb{F}=\{\mathcal{F}_t\}_{t\geq0}$, be a complete filtered probability space satisfying the usual conditions. Assume that the filtration is rich enough to carry the underlying state process $X=(X_t,\mathcal{F}_t)$ and a Poisson process $N=(N_t,\mathcal{F}_t)$. We assume that the process $N$ has jump intensity $\lambda$ and that it is independent of $X$. The process $X$ is a linear diffusion evolving on $\mathbf{R}_+$ with the infinitesimal generator $\mathcal{A}=\frac{1}{2}\beta^2(x)\frac{d^2}{dx^2}+\alpha(x)\frac{d}{dx}$ and initial state $x$. In what follows, we assume that the functions $\alpha$ and $\beta$ are continuous and that the process does not die inside the state space. The densities of the speed measure $m$ and the scale function $S$ of $X$ are defined, respectively, as $m'(x)=\frac{2}{\beta^2(x)}e^{B(x)}$ and  $S'(x)= e^{-B(x)}$ for all $x \in \mathbf{R}_+$, where $B(x):=\int^x \frac{2\alpha(y)}{\beta^2(y)}dy$.

We denote as, respectively, $\psi_r>0$ and $\varphi_r>0$ the increasing and decreasing solution of the ODE $\mathcal{A}u=ru$, where $r>0$, defined on the domain of the characteristic operator of $X$. By posing appropriate boundary conditions depending on the boundary classification of $X$, the functions $\psi_r$ and $\varphi_r$ are defined uniquely up to a multiplicative constant and can be identified as the minimal $r$-excessive functions. To impose the boundary conditions, we assume, using the terminology of \cite{BoSa}, that the boundaries $0$ and $\infty$ are natural. This means that 
\begin{align}\label{boundary} 
\lim_{x\rightarrow 0+} \psi_r(x) = \lim_{x\rightarrow 0+} \frac{\psi_r'(x)}{S'(x)} &= 0, \ \lim_{x\rightarrow \infty} \psi_r(x) = \lim_{x\rightarrow \infty} \frac{\psi_r'(x)}{S'(x)} = \infty, \\
\lim_{x\rightarrow 0+} \varphi_r(x) = \lim_{x\rightarrow 0+} \frac{\varphi_r'(x)}{S'(x)} &= \infty, \ \lim_{x\rightarrow \infty} \varphi_r(x) = \lim_{x\rightarrow \infty} \frac{\varphi_r'(x)}{S'(x)} = 0.  \nonumber
\end{align}  
Finally, we denote as $\mathbf{P}_x$ the probability measure $\mathbf{P}$ conditioned on the initial state $x$ and as $\mathbf{E}_x$ the expectation with respect to $\mathbf{P}_x$.

For $r>0$, we denote as $L^r_1$ the class of real valued functions on $\mathbf{R}_+$ satisfying the condition $\mathbf{E}_x\left[ \int_0^\infty e^{-rt}|f(X_t)|dt \right]<\infty$. For a function $f \in L^r_1$, the resolvent $R_r f$ is defined as
\[ (R_r f)(x)= \mathbf{E}_x\left[ \int_0^\infty e^{-rt}f(X_t)dt \right]. \]
It is well known, see, e.g., \cite{BoSa}, that for the considered class of diffusion processes, the resolvent operator can be expressed as 
\begin{align}\label{Resolvent}
(R_r h)(x) = B_r^{-1}\left( \varphi_r(x)\int_0^x \psi_r(z)h(z)m'(z)dz + \psi_r(x)\int_x^\infty \varphi_r(z)h(z)m'(z)dz  \right).
\end{align}
where $B_r = \frac{\psi'_r(x)}{S'(x)}\varphi_r(x) - \frac{\varphi'_r(x)}{S'(x)}\psi_r(x) $ is the Wronskian determinant. We also point out that the resolvent operator satisfies the following resolvent equation
\[ R_q - R_p + (q-p)R_qR_p = 0, \]
for all $q,p>0$.
\begin{remark}\label{Resolvent Remark}
Let $h \in L_1^r$. For an arbitrary $y>0$, define the functions $\check{h}_y$ and $\hat{h}_y$ as
\begin{align*}
\check{h}_y(x) = h(x)\mathbf{1}_{\{x<y\}}, \ \hat{h}_y(x) = h(x)\mathbf{1}_{\{x\geq y\}},
\end{align*}
for all $x \in (0,\infty)$. Then, obviously, $\check{h}_y, \hat{h}_y \in L_1^r$ and $h = \check{h}_y + \hat{h}_y$ for all $y \in (0,\infty)$. Furthermore, 
\begin{displaymath}
(R_r \check{h}_y)(x) =
\begin{cases}
B_r^{-1}\left(\varphi_r(x)\int_0^x \psi_r(z)h(z)m'(z)dz + \psi_r(x)\int_x^y \varphi_r(z)h(z)m'(z)dz \right), & x < y, \\
B_r^{-1}\varphi_r(x)\int_0^y \psi_r(z)h(z)m'(z)dz, & x \geq y,
\end{cases} 
\end{displaymath}
and
\begin{displaymath}
(R_r \hat{h}_y)(x) =
\begin{cases}
B_r^{-1}\psi_r(x)\int_y^\infty \varphi_r(z)h(z)m'(z)dz, & x < y, \\
B_r^{-1}\left(\varphi_r(x)\int_y^x \psi_r(z)h(z)m'(z)dz + \psi_r(x)\int_x^\infty \varphi_r(z)h(z)m'(z)dz \right), & x \geq y.
\end{cases} 
\end{displaymath}
\end{remark}
The following auxiliary result can be readily verified using the representation \eqref{Resolvent} and the conditions \eqref{boundary}.
\begin{lemma}\label{Lemma001}
Let $r>0$ and $h \in L_1^{r}$. Then 
\begin{align*}
(R_{r}h)'(x) - (R_{r}h)(x)\frac{\psi_{r}'(x)}{\psi_{r}(x)} &= - \frac{S'(x)}{\psi_r(x)}\int_0^x \psi_r(z)h(z)m'(z)dz, \\
\frac{\varphi_r'(x)}{S'(x)}(R_rh)(x) - \frac{(R_rh)'(x)}{S'(x)}\varphi_r(x) &= - \int_x^\infty \varphi_r(z) h(z) m'(z) dz, \\
\frac{\psi_r'(x)}{S'(x)}(R_rh)(x) - \frac{(R_rh)'(x)}{S'(x)}\psi_r(x) &=  \int_0^x \psi_r(z) h(z) m'(z) dz. \\
\end{align*}

\end{lemma}

\section{The Entry Problem}

We are now in position to state our entry problem. Let $(T_j)$ denote the jump times of the Poisson process $N$ with the convention $T_0=0$. Denote as $\bar{\tau}$ an arbitrary sequence of stopping times $(\tau_j)$ taking values in $[0,\infty]$ satisfying the constraint $\tau_j \leq \sigma_j \leq \tau_{j+1}$ for all $j$, where $\sigma_j = \inf \{ T_k \ : \ \tau_j < T_k  \}$. Let $h$ be a measurable function, $C$ and $K$ two non-negative constants, and $r>0$ the constant rate of discounting. Motivated by the discussion in the introduction, consider the following multiple entry problems:
\begin{align}\label{idle problem}
V_i(x) = \sup_{\bar{\tau}} \mathbf{E}_x\left[ \sum_{j=1}^M \left(\int_{\tau_j}^{\sigma_j} e^{-rt}(h(X_t)-C)dt-e^{-r\tau_j}K\right) \right]
\end{align}
and
\begin{align}\label{active problem}
V_a(x) = \sup_{\bar{\tau}} \mathbf{E}_x\left[\int_0^{T_1} e^{-rt}(h(X_t)-C)dt + \sum_{j=1}^M \left(\int_{\tau_j}^{\sigma_j} e^{-rt}(h(X_t)-C)dt-e^{-r\tau_j}K \right)\right]. 
\end{align}
where $M\sim \Geom(p)$ and independent of $X$. Here, the random variable $M$ is realized as follows: we attach an independent Bernoulli trial with success probability $p$ to each jump time of $N$. When the first failure occurs, the re-entry possibility expires permanently. This problem was studied in \cite{Wang2005} in the case where the diffusion $X$ is a geometric Brownian motion and no catastrophes occur, i.e., $p=1$. 

We study the problems \eqref{idle problem} and \eqref{active problem} under the following assumptions.

\begin{assumption}\label{SA}
Let the function $h$ be $L^1_r$, continuous, non-decreasing, non-constant and satisfy $h(0)=0$. Furthermore, let at least one of the constants $C$ and $K$ be strictly positive.    
\end{assumption}

It is helpful to write the value functions \eqref{idle problem} and \eqref{active problem} as infinite sums instead of sums of random length. To this end, we use the well-known thinning procedure of Poisson processes, see, e.g., \cite{Serfozo}. As we have labeled the jump times of $N$ with outcomes of independent Bernoulli trials, we can split the process $N$ into two independent Poisson processes $N^0$ and $N^1$ corresponding to outcomes $0$ and $1$, respectively. Moreover, the intensities of $N^0$ and $N^1$ are $\lambda(1-p)$ and $\lambda p$, respectively. Now, denote the jump times of $N^1$ as $T_j^1$. These jump times correspond to events where further entries are still possible. On the other hand, the first jump of $N^0$ will terminate the whole investment opportunity. Denote this jump time as $T^0_1$

Denote as $\bar{\tau}^1$ an arbitrary sequence of stopping times $(\tau_j)$ taking values in $[0,\infty]$ satisfying the constraint $\tau_j \leq \sigma^1_j \leq \tau_{j+1}$ for all $j$, where $\sigma^1_j = \inf \{ T^1_k \ : \ \tau_j < T^1_k  \}$. Then we have the following lemma.

\begin{lemma}\label{MainLemma}
The value functions \eqref{idle problem} and \eqref{active problem} can be expressed as
\begin{align}\label{idle problem rev}
V_i(x) = \sup_{\bar{\tau}^1} \mathbf{E}_x\left[ \sum_{j=1}^\infty \left(\int_{\tau_j}^{\sigma_j^1} e^{-(r+(1-p)\lambda)t}(h(X_t)-C)dt-e^{-(r+(1-p)\lambda)\tau_j}K\right)\right]
\end{align}
and 
\begin{align}\label{active problem rev}
V_a(x) = \sup_{\bar{\tau}^1} \mathbf{E}_x&\left[\int_0^{T^1_1} e^{-(r+(1-p)\lambda)t}(h(X_t)-C)dt \right. \\&+ \left.\sum_{j=1}^\infty \left(\int_{\tau_j}^{\sigma^1_j} e^{-(r+(1-p)\lambda)t}(h(X_t)-C)dt-e^{-(r+(1-p)\lambda)\tau_j}K \right)\right] \nonumber. 
\end{align}
\end{lemma}

\begin{proof}
Since the variable $T^0_1 \sim \Exp((1-p)\lambda)$ and it is independent of $X$, we can use the additive functional $A_t = (1-p)\lambda t$ to impose an additional killing rate of $dA_t = (1-p)\lambda$ on the process $X$, for details, see, e.g., \cite{BoSa}. This results in the total killing rate of $r+(1-p)\lambda$. 
\end{proof}

\begin{remark}
The previous lemma shows that the jumps of the process $N^1$ are the forced exit times for which further re-entry opportunities remain. To simplify the notation, we denote from now on the jump times of $N^1$ simply as $T_j$ for which $T_{j+1} - T_{j} \sim \Exp(\lambda p)$ and the stopping times $\sigma^1_j$ as $\sigma_j$ for all $j\geq 1$.  
\end{remark}

\begin{remark}
We point out following \cite{Wang2005} that the optimization problems \eqref{idle problem rev} and \eqref{active problem rev} are well defined. For every sequence $\bar{\tau}$ of stopping times, we have 
\begin{align*}
\sum_{j=1}^\infty &\left| \int_{\tau_j}^{\sigma_j} e^{-(r+(1-p)\lambda)t}(h(X_t)-C)dt-e^{-(r+(1-p)\lambda)\tau_j}K \right| \\
& \leq \int_0^\infty (h(X_t)+c)dt + \sum_{j=1}^\infty e^{-(r+(1-p)\lambda)\tau_j}K.
\end{align*} 
\end{remark}
Since $\tau_j \geq T_{j-1}$, we find that 
\begin{align*}
\sum_{j=1}^\infty e^{-(r+(1-p)\lambda)\tau_j}K \leq K\sum_{j=1}^\infty e^{-(r+(1-p)\lambda)T^1_{j-1}} = K\sum_{j=1}^\infty \left(\frac{\lambda p}{r+\lambda}\right)^{j-1}= K \frac{r+\lambda}{r+(1-p)\lambda}.
\end{align*} 

\begin{remark}
We observe from Lemma \ref{MainLemma} that as the parameter $p$ approaches zero, the total killing rate $r+(1-p)\lambda$ increases which, in turn, decreases the value. Thus we can say that generally under our assumptions, increased catastrophe risk (as measured by decreasing $p$) decreases the values $G_i$ and $G_a$. 
\end{remark}

\section{The Solution}

\subsection{Necessary Conditions}

We derive the candidate solutions $G_i$ and $G_a$ in a partly heuristic way. We start by making the working assumption that the continuation region for both active and idle problem is the interval $(0,x^*)$ for some threshold $x^*$. Consider first the idle problem and assume that $x<x^*$. Then it is optimal to wait. On an infinitesimal time period $dt$, the process $N$ will jump with probability $\lambda dt$. Given there is a jump, success in the Bernoulli experiment implies that further entry possibilities remain. By the memoryless property of geometric distribution, the value of the remaining entry possibilities is $G_i(x)$. On the other hand, if the Bernoulli experiment fails, there is no further entry possibilities and the remaining value is zero. Based on these observations, we expect that the function $G_i$ satisfies the following partly heuristic computation
\begin{align}\label{G_i lower}
G_i(x)	&= \lambda dt (pG_i(x) + (1-p)\cdot 0) + (1-\lambda dt) \mathbf{E}_x\left[ e^{-r dt} G_i(X_{dt}) \right] \nonumber \\ 
				&=\lambda pG_i(x)dt + (1-\lambda dt)(G_i(x) + (\mathcal{A}-r)G_i(x)dt) \\
				&= G_i(x) + \left[(\mathcal{A}-r)G_i(x) + \lambda pG_i(x) - \lambda G_i(x) \right]dt, \nonumber
\end{align}
here, we have neglected the $dt^2$-term. Thus we expect that the candidate $G_i$ satisfies
\[ (\mathcal{A}- (r+(1-p)\lambda))G_i(x)=0, \]
for all $x\in(0,x^*)$. When $x\geq x^*$, the it is optimal to make an entry, that is, to pay the cost $K$ and become active. Therefore, we expect that
\[ G_i(x) = G_a(x) - K, \]
for all $x \in [x^*,\infty)$.

Consider now the active problem. Now the decision maker can do nothing before the first forced exit. Let $x>0$. By reasoning similarly as in \eqref{G_i lower}, we expect that the function $G_a$ satisfies the following, again, partly heuristic computation
\begin{align*} 
G_a(x) 	&= (h(x)-C)dt + \lambda dt (pG_i(x) + (1-p)\cdot 0) + (1-\lambda dt) \mathbf{E}_x\left[ e^{-r dt} G_a(X_{dt}) \right] \\
				&= G_a(x) + \left[ h(x)-C + (\mathcal{A}-r)G_a(x) +\lambda(pG_i(x)-G_a(x)) \right]dt.
\end{align*}
Therefore, we expect that the candidate $G_a$ satisfies
\begin{align*}
(\mathcal{A}-r)G_a(x)+\lambda(pG_i(x)-G_a(x))+h(x)-C=0,
\end{align*}
for all $x>0$. Summarizing, the task is now to solve the following coupled free boundary problem: Find continuously differentiable functions $G_i$ and $G_a$ and threshold $x^*>0$ such that
\begin{equation}\label{FBB}
\begin{cases}
(\mathcal{A}-(r+(1-p)\lambda)G_i(x)=0, & x \leq x^*, \\
G_i(x) = G_a(x)-K, & x \geq x^*, \\ 
(\mathcal{A}-r)G_a(x)+\lambda(pG_i(x)-G_a(x))+h(x)-C=0, & x > 0.
\end{cases}
\end{equation}
By the definitions \eqref{idle problem} and \eqref{active problem}, we also expect that the functions $G_i$ and $G_a$ satisfy the growth condition
\begin{align}\label{Growth} 
G_i(x) \leq (R_{r+(1-p)\lambda} h)(x), \quad G_a(x) \leq (R_{r+\lambda}h)(x) + (R_{r+(1-p)\lambda}h)(x), 
\end{align}
for all $x \in (0,\infty)$. 
\begin{remark}
In \cite{Wang2005} it is required that $G_a\in C^2$. This additional smoothness requirement is not necessary for the results to hold as we will see in the next section.
\end{remark}
Since we are looking for a function $G_i$ that is bounded in the origin, we find from the first equation of \eqref{FBB} that
$G_i(x) = c_{i,1} \psi_{r+(1-p)\lambda}(x)$ for all $x\leq x^*$, where $c_{i,1}$ is a constant. Using this, we can rewrite the third equation of \eqref{FBB} as
\[ (\mathcal{A}-(r+\lambda))G_a(x)=-(\lambda p c_{i,1}\psi_{r+(1-p)\lambda}(x)+h(x)-C), \] 
for all $x\leq x^*$. Solutions to this ODE can be expressed as
\[ (R_{r+\lambda}h)(x)-\frac{C}{r+\lambda}+\lambda p c_{i,1} (R_{r+\lambda}\psi_{r+(1-p)\lambda})(x) + c_{a,1}\psi_{r+\lambda}(x) + c_{a,2}\varphi_{r+\lambda}(x).  \]
Again, we are looking for a solution that is bounded in the origin, so we find that $c_{a,2}=0$. Since the boundary $\infty$ is natural for the diffusion $X$, we can use \cite[Lemma 2.1]{Lempa2012} to rewrite the solution as
\[ (R_{r+\lambda}h)(x)-\frac{C}{r+\lambda}+c_{i,1}\psi_{r+(1-p)\lambda}(x) + c_{a,1}\psi_{r+\lambda}(x),   \]
for all $x\leq x^*$.

Next, we study the solutions on the interval $(x^*,\infty)$. Using the two last equations in \eqref{FBB}, we find that
\[ (\mathcal{A}-(r+(1-p)\lambda))G_i(x)-K(r+\lambda)+h(x)-C=0 \]
for all $x\geq x^*$. The solutions to this ODE can be written as
\[ (R_{r+(1-p)\lambda}h)(x) -\frac{C+K(r+\lambda)}{r+(1-p)\lambda}+d_{i,1}\psi_{r+(1-p)\lambda}(x)+d_{i,2}\varphi_{r+(1-p)\lambda}(x). \]
The growth condition \eqref{Growth} implies that $d_{i,1}=0$. Furthermore, we observe that
\[ (R_{r+(1-p)\lambda}h)(x)-\frac{C+K(r+\lambda)}{{r+(1-p)\lambda}} = (R_{r+(1-p)\lambda} h_C)(x), \]
where $h_C(x)= h(x) - (C+K(r+\lambda))$ for all $x \in (0,\infty)$. Summarising, we have the following candidate solutions
\begin{equation}\label{Candidate}
\begin{split}
G_i(x) &=
\begin{cases}
c_{i,1} \psi_{r+(1-p)\lambda}(x), & x < x^*, \\
(R_{r+(1-p)\lambda}h_C)(x)+d_{i,2}\varphi_{r+(1-p)\lambda}(x), & x \geq x^*,
\end{cases} \\
G_a(x) &=  
\begin{cases}
(R_{r+\lambda}h)(x)-\frac{C}{r+\lambda}+c_{i,1}\psi_{r+(1-p)\lambda}(x) + c_{a,1}\psi_{r+\lambda}(x), & x < x^*, \\
(R_{r+(1-p)\lambda}h)(x) -\frac{C+K p \lambda}{r+(1-p)\lambda}+d_{i,2}\varphi_{r+(1-p)\lambda}(x), & x \geq x^*.
\end{cases}
\end{split}
\end{equation}

To determine the unknown constants $c_{i,1}$, $c_{a,1}$, $d_{i,2}$, and the threshold $x^*$, we use value-matching and smooth-pasting conditions. First, since the candidate value $G_i$ and its derivative are continuous over $x^*$, we find that the conditions 
\begin{equation*}
\begin{cases}
c_{i,1} \psi_{r+(1-p)\lambda}(x^*) - d_{i,2}\varphi_{r+(1-p)\lambda}(x^*) &= (R_{r+(1-p)\lambda}h_C)(x^*), \\
c_{i,1} \psi_{r+(1-p)\lambda}'(x^*) - d_{i,2}\varphi_{r+(1-p)\lambda}'(x^*) &= (R_{r+(1-p)\lambda}h_C)'(x^*).
\end{cases}
\end{equation*}
must hold. This implies that
\begin{align}\label{cdConst}
c_{i,1} &= B_{r+(1-p)\lambda}^{-1}\left( \frac{(R_{r+(1-p)\lambda} h_C)'(x^*)}{S'(x^*)}\varphi_{r+(1-p)\lambda}(x^*) -  \frac{\varphi_{r+(1-p)\lambda}'(x^*)}{S'(x^*)}(R_{r+(1-p)\lambda} h_C)(x^*) \right), \\
d_{i,2} &= -B_{r+(1-p)\lambda}^{-1}\left( \frac{\psi_{r+(1-p)\lambda}'(x^*)}{S'(x^*)}(R_{r+(1-p)\lambda} h_C)(x^*) -  \frac{(R_{r+(1-p)\lambda} h_C)'(x^*)}{S'(x^*)}\psi_{r+(1-p)\lambda}(x^*) \right).\nonumber
\end{align}
and, consequently, by using Lemma \ref{Lemma001} that
\begin{align*}
c_{i,1} &= B_{r+(1-p)\lambda}^{-1} \int_{x^*}^\infty \varphi_{r+(1-p)\lambda}(z) h_C(z) m'(z) dz , \\
d_{i,2} &= -B_{r+(1-p)\lambda}^{-1} \int_{0}^{x^*} \psi_{r+(1-p)\lambda}(z) h_C(z) m'(z) dz .
\end{align*}
Now we can solve the constant $c_{a,1}$ by substitution. Indeed, after an application of the resolvent equation, a simplification yields
\begin{align}\label{c2Const}
c_{a,1}=\frac{\lambda p(R_{r+\lambda}R_{r+(1-p)\lambda}h_C)(x^*)-c_{i,1}\psi_{r+(1-p)\lambda}(x^*)+d_{i,2}\varphi_{r+(1-p)\lambda}(x^*)}{\psi_{r+\lambda}(x^*)}.
\end{align}

To characterise the threshold $x^*$, we use the smoothness of $G_a$ over $x^*$. First, an application of the resolvent equation yields
\begin{equation}
\lambda p(R_{r+\lambda}R_{r+(1-p)\lambda}h_C)(x^*) - c_{a,1}\psi_{r+\lambda}'(x^*) = c_{i,1}\psi_{r+(1-p)\lambda}'(x^*) - d_{i,2}\varphi_{r+(1-p)\lambda}'(x^*).
\end{equation}
Using the expression for $c_{a,1}$, we rewrite this condition as
\begin{align}\label{Necessary001}
\lambda p&(R_{r+\lambda}R_{r+(1-p)\lambda} h_C)'(x^*) - \lambda p(R_{r+\lambda}R_{r+(1-p)\lambda} h_C)(x^*)\frac{\psi_{r+\lambda}'(x^*)}{\psi_{r+\lambda}(x^*)} \nonumber = \\
&c_{i,1}\left( \psi_{r+(1-p)\lambda}'(x^*) - \psi_{r+(1-p)\lambda}(x^*) \frac{\psi_{r+\lambda}'(x^*)}{\psi_{r+\lambda}(x^*)} \right) - \\ & d_{i,2}\left(\varphi_{r+(1-p)\lambda}'(x^*) - \varphi_{r+(1-p)\lambda}(x^*) \frac{\psi_{r+\lambda}'(x^*)}{\psi_{r+\lambda}(x^*)} \right) . \nonumber
\end{align} 
We find using \cite[Lemma 2.1]{Lempa2012} and Lemma \ref{Lemma001} that
\begin{align*}
\varphi_{r+(1-p)\lambda}&'(x^*) - \varphi_{r+(1-p)\lambda}(x^*) \frac{\psi_{r+\lambda}'(x^*)}{\psi_{r+\lambda}(x^*)} = \\  \lambda p(R_{r+\lambda}&\varphi_{r+(1-p)\lambda})'(x^*) - 
	\lambda p(R_{r+\lambda}\varphi_{r+(1-p)\lambda})(x^*) \frac{\psi_{r+\lambda}'(x^*)}{\psi_{r+\lambda}(x^*)} - \\
	&\frac{S'(x^*)}{\psi_{r+\lambda}(x^*)}\lambda p \int_0^{x^*}\psi_{r+\lambda}(z)\varphi_{r+(1-p)\lambda}(z)m'(z)dz;
\end{align*}
we can handle the other term on the right hand side of \eqref{Necessary001} similarly. By applying Lemma \ref{Lemma001} also to the left hand side, we can express the condition \eqref{Necessary001} as 
\begin{align}\label{Necessary004}
\nonumber B_{r+(1-p)\lambda} \int_{0}^{x^*} &\psi_{r+\lambda}(z)(R_{r+(1-p)\lambda}h_C)(z)m'(z)dz = \\
& \int_{x^*}^\infty \varphi_{r+(1-p)\lambda}(z)h_C(z)m'(z)dz \int_0^{x^*}\psi_{r+\lambda}(z)\psi_{r+(1-p)\lambda}(z)m'(z)dz \\  \nonumber & +  \int_0^{x^*} \psi_{r+(1-p)\lambda}(z)h_C(z)m'(z)dz \int_0^{x^*}\psi_{r+\lambda}(z)\varphi_{r+(1-p)\lambda}(z)m'(z)dz.
\end{align}
Next, apply the representation \eqref{Resolvent} and change the order of integration; this yields the expression 
\begin{align}\label{Necessary005}
\int_0^{x^*} h_C(y)&\left( \psi_{r+(1-p)\lambda}(y)\int_{0}^{y}\psi_{r+\lambda}(z)\varphi_{r+(1-p)\lambda}(z) m'(z)dz \right. \\ \nonumber &- \left. \varphi_{r+(1-p)\lambda}(y)\int_{0}^{y}\psi_{r+\lambda}(z)\psi_{r+(1-p)\lambda}(z) m'(z)dz \right)m'(y)dy = 0.
\end{align}
Using the properties of functions $\psi_\cdot$ and $\varphi_\cdot$, it is elementary to verify that
\begin{align*}
\lambda p \int_{0}^{y}\psi_{r+\lambda}(z)\varphi_{r+(1-p)\lambda}(z) m'(z)dz &= \frac{\psi_{r+\lambda}'(y)}{S'(y)}\varphi_{r+(1-p)\lambda}(y) - \frac{\varphi_{r+(1-p)\lambda}'(y)}{S'(y)}\psi_{r+\lambda}(y),  \\
\lambda p \int_{0}^{y}\psi_{r+\lambda}(z)\psi_{r+(1-p)\lambda}(z) m'(z)dz &= \frac{\psi_{r+\lambda}'(y)}{S'(y)}\psi_{r+(1-p)\lambda}(y) - \frac{\psi_{r+(1-p)\lambda}'(y)}{S'(y)}\psi_{r+\lambda}(y).
\end{align*}
Finally, using these identities, a round of simplification shows that the condition \eqref{Necessary005} can be rewritten as
\begin{align}\label{Necessary006} 
\int_0^{x^*} \psi_{r+\lambda}(y)&h_C(y)m'(y)dy = 0 \\ &\Leftrightarrow \nonumber \int_0^{x^*} \psi_{r+\lambda}(y)h(y)m'(y)dy = \left(K + \frac{C}{r+\lambda}\right)\frac{\psi_{r+\lambda}'(x^*)}{S'(x^*)}.
\end{align}
By our assumptions, it is obvious that there is a unique $x^*$ satisfying this condition. Summarising, we collect the results derived in this section to the following lemma.

\begin{lemma}\label{CandiLemma}
Let Assumptions \ref{SA} hold. Then the free boundary problem \eqref{FBB} has a unique solution $(G_i,G_a,x^*)$, where the functions $G_i$ and $G_a$ are defined in \eqref{Candidate} such that \eqref{cdConst} and \eqref{c2Const} hold and the threshold $x^*$ is characterised by \eqref{Necessary006}. 
\end{lemma}

\begin{remark}
We notice from the condition \eqref{Necessary006} that the optimal entry threshold does not depend on the parameter $p$. This is a somewhat remarkable result. Indeed, if the decision maker is idle, her decision rule is independent of the success probability of the Bernoulli trial at the next jump time of the Poisson process $N$. Even if this probability is very low, meaning that it is very likely that there will be no further re-entry opportunities after the next jump of $N$, it is still optimal for the decision maker to act as if this probability is equal to one.     
\end{remark}

\subsection{Sufficient Conditions} The purpose of this section is to prove that our candidate solution described in Lemma \ref{CandiLemma} is the solution of the main problem. The next result is our main theorem.

\begin{theorem}\label{MainThrm} Let Assumption \ref{SA} hold.
\begin{itemize}
\item[(1)] If $\lim_{x\rightarrow\infty} h(x) \leq C+(r+\lambda)K$, then $V_i \equiv 0$ and, and it is optimal never to make an entry.
\item[(2)] Let $\lim_{x\rightarrow\infty} h(x) > C+(r+\lambda)K$ and $(G_i,G_a,x^*)$ be the solution given in Lemma \ref{CandiLemma}. Then $V_i(x)=G_i(x)$ for all $x\in(0,\infty)$. Furthermore, the sequence of optimal entry times is recursively defined as
\begin{align}\label{OptStopTime}
\tau_j^* = \inf \left\{ \tau \geq \sigma^*_{j-1} \ : \ X_t \geq x^* \right\}
\end{align}
for every $j\geq1$, where $\sigma^*_0=0$ and $\sigma^*_j = \inf\{ T_k \ : \ \tau^*_j < T_k \}$ for $j\geq1$. 
\end{itemize}
\end{theorem}

Before proving the main theorem, we establish some auxiliary results needed in the proof. The first one shows essentially that the candidates $G_i$ and $G_a$ satisfy Bellman's principle.

\begin{lemma}\label{BellmanLemma}
Let $\tau$ be a stopping time and $\sigma := \inf \{ T_j \ : \ \tau < T_j \}$ where the inter-arrival times $T_{j+1}-T_j\sim \Exp(\lambda p)$. Then
\[ e^{-(r+(1-p)\lambda)\tau}G_a(X_\tau) = \mathbf{E}_{X_\tau}\left[ \int_\tau^\sigma e^{-(r+(1-p)\lambda)t}(h(X_t)-C)dt + e^{-(r+(1-p)\lambda)\sigma}  G_i(X_\sigma) \right]. \]
\end{lemma}

\begin{proof}
Strong Markov property of $X$ coupled with the memoryless property of the exponential distribution guarantees that it is sufficient to show that
\begin{align*}
G_a(x) = \mathbf{E}_x\left[ \int_0^U e^{-(r+(1-p)\lambda)t}(h(X_t)-C)dt + e^{-(r+(1-p)\lambda)U} G_i(X_U) \right],
\end{align*}
where $U\sim \Exp(\lambda p)$, for all $x\in(0,\infty)$. Using the resolvent equation, this can be rewritten as
\begin{align*}
G_a(x) = (R_{r+\lambda}\bar{h})(x) + \lambda p (R_{r+\lambda}G_i)(x),
\end{align*}
where $\bar{h}(x)=h(x)-C$. First, let $x < x^*$. Using Remark \ref{Resolvent Remark}, we find that
\begin{align*}
\lambda p (R_{r+\lambda}G_i)(x) &= \lambda p ((R_{r+\lambda}\hat{G_i}_{x^*})(x) + (R_{r+\lambda}\check{G_i}_{x^*})(x)) = \\
																&= \lambda p B_{r+\lambda}^{-1} \psi_{r+\lambda}(x)\int_{x^*}^\infty \varphi_{r+\lambda}(z)G_i(z)m'(z)dz  \\
																+ \lambda p B_{r+\lambda}^{-1}&\left( \varphi_{r+\lambda}(x)\int_0^x \psi_{r+\lambda}(z)G_i(z)m'(z)dz + \psi_{r+\lambda}(x)\int_x^{x^*} \varphi_{r+\lambda}(z)G_i(z)m'(z)dz \right).
\end{align*}
By substituting the expression for $G_i$ from \eqref{Candidate}, we obtain by using first \cite[Lemma 2.1]{Lempa2012}, then the optimality condition \eqref{Necessary004} and finally \cite[Lemma 2.1]{Lempa2012} again
\begin{align*}
\lambda p (R_{r+\lambda}G_i)&(x) = c_{i,1} \psi_{r+(1-p)\lambda}(x) - \lambda p B_{r+\lambda}^{-1} \psi_{r+\lambda}(x) \\ 
					&\times \int_{x^*}^\infty \varphi_{r+\lambda}(z) \left( (R_{r+(1-p)\lambda}h_C)(z) + d_{i,2}\varphi_{r+(1-p)\lambda}(z) - c_{i,1}\psi_{r+(1-p)\lambda}(z)   \right)m'(z)dz \\
					&= c_{i,1} \psi_{r+(1-p)\lambda}(x) + \psi_{r+\lambda}(x) \\ 
					&\times \frac{\lambda p (R_{r+\lambda}R_{r+(1-p)\lambda}h_C)(x^*) + d_{i,2} \varphi_{r+(1-p)\lambda}(x^*) - c_{i,1} \psi_{r+(1-p)\lambda}(x^*)}{\psi_{r+\lambda}(x^*)} \\
					&= G_a(x) - (R_{r+\lambda}\bar{h})(x).	
\end{align*}
Now, let $x \geq x^*$. Again, by using Remark \ref{Resolvent Remark}, we find that
\begin{align*}
\lambda p (R_{r+\lambda}G_i)(x) &= \lambda p ((R_{r+\lambda}\hat{G_i}_{x^*})(x) + (R_{r+\lambda}\check{G_i}_{x^*})(x))  \\
																= \lambda p B_{r+\lambda}^{-1} &\left( \varphi_{r+\lambda}(x)\int_{x^*}^x \psi_{r+\lambda}(z)G_i(z)m'(z)dz + \psi_{r+\lambda}(x)\int_x^\infty \varphi_{r+\lambda}(z)G_i(z)m'(z)dz \right) \\
																&+ \lambda p B_{r+\lambda}^{-1}	\varphi_{r+\lambda}(x)\int_0^{x^*} \psi_{r+\lambda}(z)G_i(z)m'(z)dz.
\end{align*}
By substituting the expression for $G_i$ from \eqref{Candidate}, we find by using first the optimality condition \eqref{Necessary004}, then \cite[Lemma 2.1]{Lempa2012}, and finally the resolvent equation, that
\begin{align*}
\lambda &p (R_{r+\lambda}G_i)(x) = \\
&\lambda p B_{r+\lambda}^{-1}\left( \varphi_{r+\lambda}(x)\int_{x^*}^x \psi_{r+\lambda}(z)\left( (R_{r+(1-p)\lambda}h_C)(z)+d_{i,2}\varphi_{r+(1-p)\lambda}(z) \right)m'(z)dz\right) + \\
&\lambda p B_{r+\lambda}^{-1}\left(  \psi_{r+\lambda}(x)\int_x^\infty \varphi_{r+\lambda}(z)\left( (R_{r+(1-p)\lambda}h_C)(z)+d_{i,2}\varphi_{r+(1-p)\lambda}(z) \right)m'(z)dz \right) + \\
& \lambda p B_{r+\lambda}^{-1}	\varphi_{r+\lambda}(x)\int_0^{x^*} \psi_{r+\lambda}(z)c_{i,1}\psi_{r+(1-p)\lambda}(z)m'(z)dz = \\
&\lambda p (R_{r+\lambda} R_{r+(1-p)\lambda}h_C)(x) + d_{i,2}\varphi_{r+(1-p)\lambda}(x) = G_a(x) - (R_{r+\lambda}\bar{h})(x).
\end{align*}
This completes the proof.
\end{proof}

\begin{remark}
In \cite{Wang2005}, this results was proved using the free boundary problem \eqref{FBB} as a variational inequality. There, the fact that the candidates solve the variational inequality was primary and the form of the actual solution was secondary. In \eqref{BellmanLemma}, we did the opposite and proved the result by the properties of the solution itself. 
\end{remark}

\begin{remark}\label{SuffiRemark}
Since the function $G_i$ is finite and $r+(1-p)\lambda$-excessive, the process $t\mapsto e^{-(r+(1-p)\lambda)t}G_i(X_t)$ is a non-negative supermartingale, see, e.g., \cite{CW}. Furthermore, let $\tau$ be a stopping time and 
\[ \tau_{x^*} = \inf \{ t\geq \tau \ : \ X_t \geq x^* \}. \]
We claim that 
\[ e^{-(r+(1-p)\lambda)\tau}G_i(X_\tau) = \mathbf{E}_{X_\tau}\left[ e^{-(r+(1-p)\lambda)\tau_{x^*}}G_i(X_{\tau_{x^*}}) \right]. \]
By strong Markov property, we can take $\tau=0$. If $x:=X_\tau \geq x^*$, the claim holds trivially. Let $x<x^*$. Since $\psi_{r+(1-p)\lambda}$ is $r+(1-p)$-harmonic and  $X$ has continuous paths, we have that
\[ \mathbf{E}_{x}\left[ e^{-(r+(1-p)\lambda){\tau_{x^*}}}G_i(X_{\tau_{x^*}}) \right] = \mathbf{E}_{x}\left[ e^{-(r+(1-p)\lambda){\tau_{x^*}}}\right] c_{i,1}\psi_{r+(1-p)\lambda}(x^*) = G_i(x);   \]
for the last equality, see, e.g., \cite{BoSa}.
\end{remark}

\begin{lemma}\label{LimitLemma}
Let $(\tau_j)$ be a sequence of stopping times such that $\tau_j\rightarrow\infty$ as $j \rightarrow \infty$. Then $\mathbf{E}_x \left[e^{-(r+(1-p)\lambda)\tau_{j}} G_i(X_{\tau_{j}})\right] \rightarrow 0$ as $j \rightarrow \infty$.
\end{lemma}

\begin{proof}
For each $j\geq 1$, we write
\begin{align}
\mathbf{E}_x \left[e^{-(r+(1-p)\lambda)\tau_{j}} G_i(X_{\tau_{j}})\right] = \psi_{r+(1-p)\lambda}(x)\hat{\mathbf{E}}_x \left[ \frac{G_i(X_{\tau_{j}})}{\psi_{r+(1-p)\lambda}(X_{\tau_{j}})}\right],
\end{align}
where $\hat{\mathbf{E}}_x$ is the expectation with respect to the probability associated with Doob's $\psi_{r+(1-p)\lambda}$-transform of $X$, see, e.g., \cite{BoSa}. We find using the representation \eqref{Resolvent} that $G_i(x)/\psi_{r+(1-p)\lambda}(x)\rightarrow 0$ as $x\rightarrow \infty$. Since the $\psi_{r+(1-p)\lambda}$-transform of $X$ is the process $X$ conditioned to exit the state space via the upper boundary $\infty$, the claim follows. 
\end{proof}

\begin{proof}[Proof of Theorem \ref{MainThrm}]
We prove first Case (2). Let $\bar{\tau}=(\tau_n)$ be an arbitrary sequence of stopping times such that $\tau_n \leq \sigma_n \leq \tau_{n+1}$ where $\sigma_n = \inf \{ T_j \ : \ \tau_n < T_j \}$ and the inter-arrival times $T_{j+1}-T_j\sim \Exp(\lambda p)$. By using the supermartingale property of $t\mapsto e^{-(r+(1-p)\lambda)t}G_i(X_t)$, the fact that $G_i(x)\geq G_a(x) - K$ for all $x \in (0,\infty)$, then Lemma \ref{BellmanLemma}, and the supermartingale property again, we obtain
\begin{align}\label{MainThrmIneq}
G_i(x) 	&\geq \mathbf{E}_x\left[ e^{-(r+(1-p)\lambda)\tau_1}G_i(X_{\tau_1}) \right] \nonumber \\
				&\geq \mathbf{E}_x\left[ e^{-(r+(1-p)\lambda)\tau_1}G_a(X_{\tau_1}) - e^{-(r+(1-p)\lambda)\tau_1}K \right]  \\
				&= \mathbf{E}_x\left[ \int_{\tau_1}^{\sigma_1} e^{-(r+(1-p)\lambda)t}(h(X_t)-C)dt + e^{-(r+(1-p)\lambda)\sigma_1} G_i(X_{\sigma_1}) - e^{-(r+(1-p)\lambda)\tau_1}K \right] \nonumber \\
				&\geq \mathbf{E}_x\left[ \int_{\tau_1}^{\sigma_1} e^{-(r+(1-p)\lambda)t}(h(X_t)-C)dt - e^{-(r+(1-p)\lambda)\tau_1}K \right]  + \mathbf{E}_x \left[e^{-(r+(1-p)\lambda)\tau_2} G_i(X_{\tau_2})\right]. \nonumber
\end{align}
By repeating this argument, we find that
\begin{align*}
G_i(x) &= \mathbf{E}_x\left[ \sum_{j=1}^n \int_{\tau_j}^{\sigma_j} e^{-(r+(1-p)\lambda)t}(h(X_t)-C)dt - e^{-(r+(1-p)\lambda)\tau_j}K  \right]  \\ &+ \mathbf{E}_x \left[e^{-(r+(1-p)\lambda)\tau_{j+1}} G_i(X_{\tau_{j+1}})\right] \\
& \geq \mathbf{E}_x\left[ \sum_{j=1}^n \int_{\tau_j}^{\sigma_j} e^{-(r+(1-p)\lambda)t}(h(X_t)-C)dt - e^{-(r+(1-p)\lambda)\tau_j}K  \right], 
\end{align*}
for all $n \geq 1$. Let $n \rightarrow \infty$ and apply Dominated Convergence Theorem. Then, by taking supremum over all $\bar{\tau}$, we obtain $G_i(x)\geq V_i(x)$ for all $x\in(0,\infty)$.

To establish the opposite inequality, let $\tau_i = \tau^*_i$, where $\tau^*_i$ is given by \eqref{OptStopTime}. We find using Remark \ref{SuffiRemark} that in this case, all inequalities in \eqref{MainThrmIneq} become equalities. Therefore
\begin{align*} 
G_i(x) = \mathbf{E}_x&\left[ \sum_{j=1}^n \int_{\tau^*_j}^{\sigma^*_j} e^{-(r+(1-p)\lambda)t}(h(X_t)-C)dt - e^{-(r+(1-p)\lambda)\tau^*_j}K  \right]  \\& + \mathbf{E}_x \left[e^{-(r+(1-p)\lambda)\tau^*_{j+1}} G_i(X_{\tau^*_{j+1}})\right],
\end{align*}
for all $j\geq1$. Clearly $\tau^*_j \geq T_{j-1}$. This implies that $\tau^*_j \rightarrow \infty$ almost surely. Using Lemma \ref{LimitLemma}, we find that
\[ \mathbf{E}_x \left[e^{-(r+(1-p)\lambda)\tau^*_{j+1}} G_i(X_{\tau^*_{j+1}})\right] \rightarrow 0, \]
as $j \rightarrow \infty$. Letting $n \rightarrow \infty$, we obtain by Dominated Convergence
\begin{align*}
G_i(x) = \mathbf{E}_x\left[ \sum_{j=1}^\infty \int_{\tau^*_j}^{\sigma^*_j} e^{-(r+(1-p)\lambda)t}(h(X_t)-C)dt - e^{-(r+(1-p)\lambda)\tau^*_j}K  \right], 
\end{align*}
which implies that $G_i(x)\leq V_i(x)$ for all $x\in(0,\infty)$. Thus $G_i = V_i$ and the sequence $\tau^*_j$ yields the optimal value.

Next, we consider Case (1). It is sufficient to show that for every $\bar{\tau}$, we have 
\[ \mathbf{E}_x\left[ \int_{\tau_j}^{\sigma_j} e^{r+(1-p)\lambda}(h(X_t)-C)dt + e^{-(r+(1-p)\lambda)\tau_j}K \right] \leq 0, \] 
for all $j$. By the strong Markov property of $X$ and the memoryless property of exponential distribution, we only need to establish that
\[ \mathbf{E}_x\left[ \int_0^U e^{-(r+(1-p)\lambda)t}(h(X_t)-C)dt \right] = (R_{r+\lambda}\bar{h})(x)\leq K.  \]
By the monotonicity of $h$, we find that
\[ (R_{r+\lambda}\bar{h})(x) \leq \frac{\lim_{x \rightarrow \infty}h(x)-C}{r+\lambda} \leq K.  \]
The proof is now complete.
\end{proof}

\begin{remark}
Under Assumptions \ref{SA}, it is possible to prove, similarly to Theorem \ref{MainThrm}, that the candidate $G_a=V_a$ and that the optimal entry threshold is $x^*$.
\end{remark}

\begin{remark}
The properties of the value functions and the optimal entry threshold with respect to the parameter $\lambda$ were studied in detail in \cite{Wang2005} in the case of GBM dynamics. For a general diffusion process, a similar analysis is a formidable task and is left for future research. 
\end{remark}

\section{Some Illustrations}

We illustrate in this section some of our results using explicit  examples. 

\subsection{Geometric Brownian Motion} Assume that the process $X$ follows a geometric Brownian motion, that is, a diffusion process with the infinitesimal generator
\[ \mathcal{A} = \frac{1}{2}\beta^2 x^2 \frac{d^2}{dx^2} + \alpha x \frac{d}{dx}. \] 
We assume that $\alpha-\frac{1}{2}\beta^2>0$. In this case, the process $X_t\rightarrow \infty$ almost surely as $t\rightarrow \infty$. It is easy to check that for an arbitrary $\rho>0$, the functions $\psi_\rho$ and $\varphi_\rho$ read as
\begin{align*} 
\psi_\rho(x) &= x^{b(\rho)}, \ b(\rho)=\left(\frac{1}{2}-\frac{\alpha}{\beta^2}\right) + \sqrt{\left(\frac{1}{2}-\frac{\alpha}{\beta^2}\right)^2 + \frac{2\rho}{\beta^2}}, \\
\varphi_\rho(x) &= x^{a(\rho)}, \ a(\rho)=\left(\frac{1}{2}-\frac{\alpha}{\beta^2}\right) - \sqrt{\left(\frac{1}{2}-\frac{\alpha}{\beta^2}\right)^2 + \frac{2\rho}{\beta^2}}. \\
\end{align*}
We verify readily that the densities $S'$ and $m'$ read as $S'(x) = x^{-\frac{2\alpha}{\beta^2}}$ and $m'(x)=\frac{2}{\beta^2 x^2}x^{\frac{2\alpha}{\beta^2}}$. Moreover, the Wronskian determinant $B_\rho = 2\sqrt{\left(\frac{1}{2}-\frac{\alpha}{\beta^2}\right)^2 + \frac{2\rho}{\beta^2}}$. 

Using this information and the formula \eqref{Necessary006}, we find that the optimal entry threshold $x^*$ is characterized by
\[ \frac{2}{\beta^2}\int_0^{x^*} z^{-(a(r+\lambda)+1)} h(z)dz = \left(K + \frac{C}{r+\lambda}\right)b(r+\lambda){x^*}^{b(r+\lambda)+\frac{2\alpha}{\beta^2}-1}, \]
which can be further simplified to
\begin{align*}
\int_0^{1} y^{-(a(r+\lambda)+1)} h(yx^*)dz = \frac{1}{2}\beta^2b(r+\lambda)\left(K + \frac{C}{r+\lambda}\right).
\end{align*}
This is the expression (4.19) in \cite{Wang2005}.

As we have observed already in the general case, the optimal entry threshold is independent of the parameter $p$. We illustrate the effect of the parameter $p$ on the value $G_i$. Let $x \leq x^*$. Straightforward integration yields
\begin{align*}
G_i(x)= \frac{2}{\beta^2 B_{r+(1-p)\lambda}}\left(\frac{x}{x^*}\right)^{b(r+(1-p)\lambda)} \int_1^\infty y^{-(b(r+(1-p)\lambda)+1)} h_C(y x^*)dy. 
\end{align*}
Since $b(r+(1-p)\lambda)>1$ for all $p\in[0,1]$, the term $\left(\frac{x}{x^*}\right)^{b(r+(1-p)\lambda)}$ becomes smaller as $p$ approaches zero. Similarly, we observe that the integral term becomes smaller as $p$ approaches zero. Finally, since the Wronskian $B_{r+(1-p)\lambda}$ increases as $p$ approaches zero, we conclude that as the probability of success $p$ becomes smaller, the value of the idle problem decreases on $(0,x^*)$. 

Let $x \geq x^*$. Straightforward integration and an application of \eqref{Resolvent} yields
\begin{align*}
G_i(x) &= (R_{r+(1-p)\lambda}h_C)(x) -\frac{2}{\beta^2 B_{r+(1-p)\lambda}}\left(\frac{x}{x^*}\right)^{a(r+(1-p)\lambda)} \int_0^1 y^{-(a(r+(1-p)\lambda)+1)}h_C(y x^*)dy \\
				&= \frac{2}{\beta^2 B_{r+(1-p)\lambda}}\left[ \int_0^1 y^{-(a(r+(1-p)\lambda)+1)}\left( h_C(yx) - \left(\frac{x}{x^*}\right)^{a(r+(1-p)\lambda)} h_C(yx^*) \right)dy \right. \\
				&\phantom{abcdefghijklmnop}+ \left. \int_1^\infty y^{-(b(r+(1-p)\lambda)+1)}h_C(yx)dy \vphantom{\left(\frac{x}{x^*}\right)^{a(r+(1-p)\lambda)}} \right].
\end{align*}
We observe by standard differentiation that in the expression above, both integrands decrease on their respective intervals as $p$ approaches zero. Summarizing, we conclude that the value $G_i$ decreases on $(0,\infty)$ as $p$ decreases. This observation is in line with the general result and illustrates that increased catastrophe risk decreases value.

To conclude, we graphically illustrate the value function $G_i$ for various values of $p$. Let $h(x)=\sqrt{x}$. It is easy to verify that the optimal entry threshold
 \[ x^* = \left[\left(\frac{1}{2}-a(r+\lambda) \right)\left(\frac{1}{2}\beta^2b(r+\lambda)\left(K + \frac{C}{r+\lambda}\right) \right) \right]^2.  \]
In Figure \ref{fig1} we illustrate the curves $x \mapsto c_{i,1}\psi_{r+(1-p)\lambda}(x)$ and and $x\mapsto (R_{r+(1-p)\lambda}h_C)(x)+d_{i,2}\varphi_{r+(1-p)\lambda}(x)$ around the optimal entry threshold $x^*$ under the parameter configuration $\alpha = 0.05$, $\beta=0.25$, $r=0.1$, $K=C=\lambda=1$, and $p=0.5$. The figure suggests that the curves tangent at $x^*$. This is in line with the smoothness requirement of $G_i$.

\begin{figure}[ht]
\caption{The curves $x \mapsto c_{i,1}\psi_{r+(1-p)\lambda}(x)$ (solid black line) and $x\mapsto (R_{r+(1-p)\lambda}h_C)(x)+d_{i,2}\varphi_{r+(1-p)\lambda}(x)$ (dashes black line). The grey dashed line is located at $x^*=5.144979$. \label{fig1}}
\includegraphics[scale=0.8]{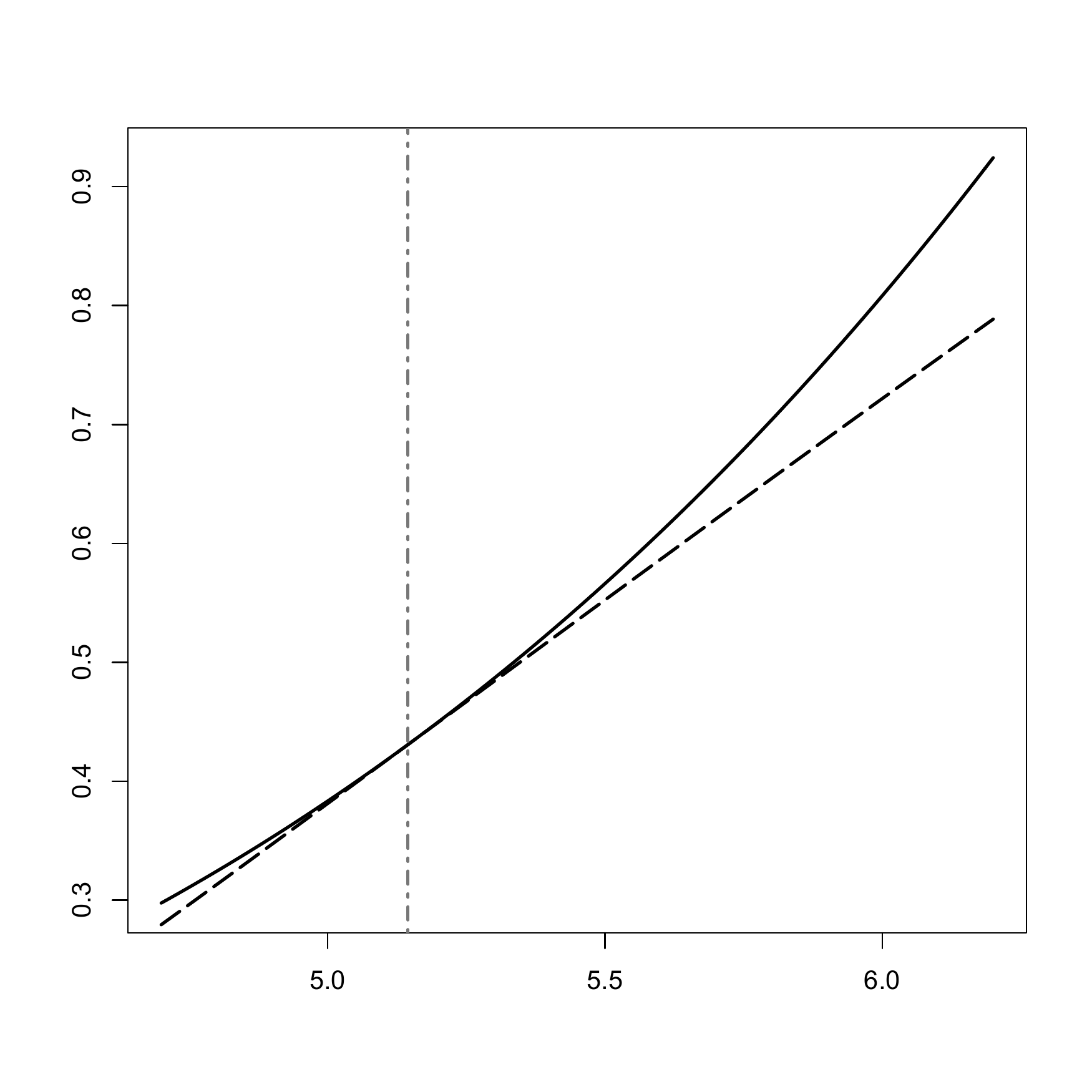}
\end{figure}

In Figure \ref{fig2} we illustrate the effect of parameter $p$ on the value $G_i$ under the parameter configuration $\alpha = 0.05$, $\beta=0.25$, $r=0.1$, and $K=C=\lambda=1$. The values of $p$ are $0.8$, $0.6$, $0.4$, and $0.2$ and the curves are colored such that the hue becomes lighter as the probability $p$ decreases. The figure shows that decreasing $p$ decreases value, this is in line with our general result.

\begin{figure}[ht]
\caption{The value $G_i$ for different values of $p$, $p=0.8,0.6,0.4,0.2$. The hue of the curve becomes lighter as the probability $p$ decreases. The grey dashed line marks the optimal entry threshold $x^*=5.144979$. \label{fig2}}
\includegraphics[scale=0.8]{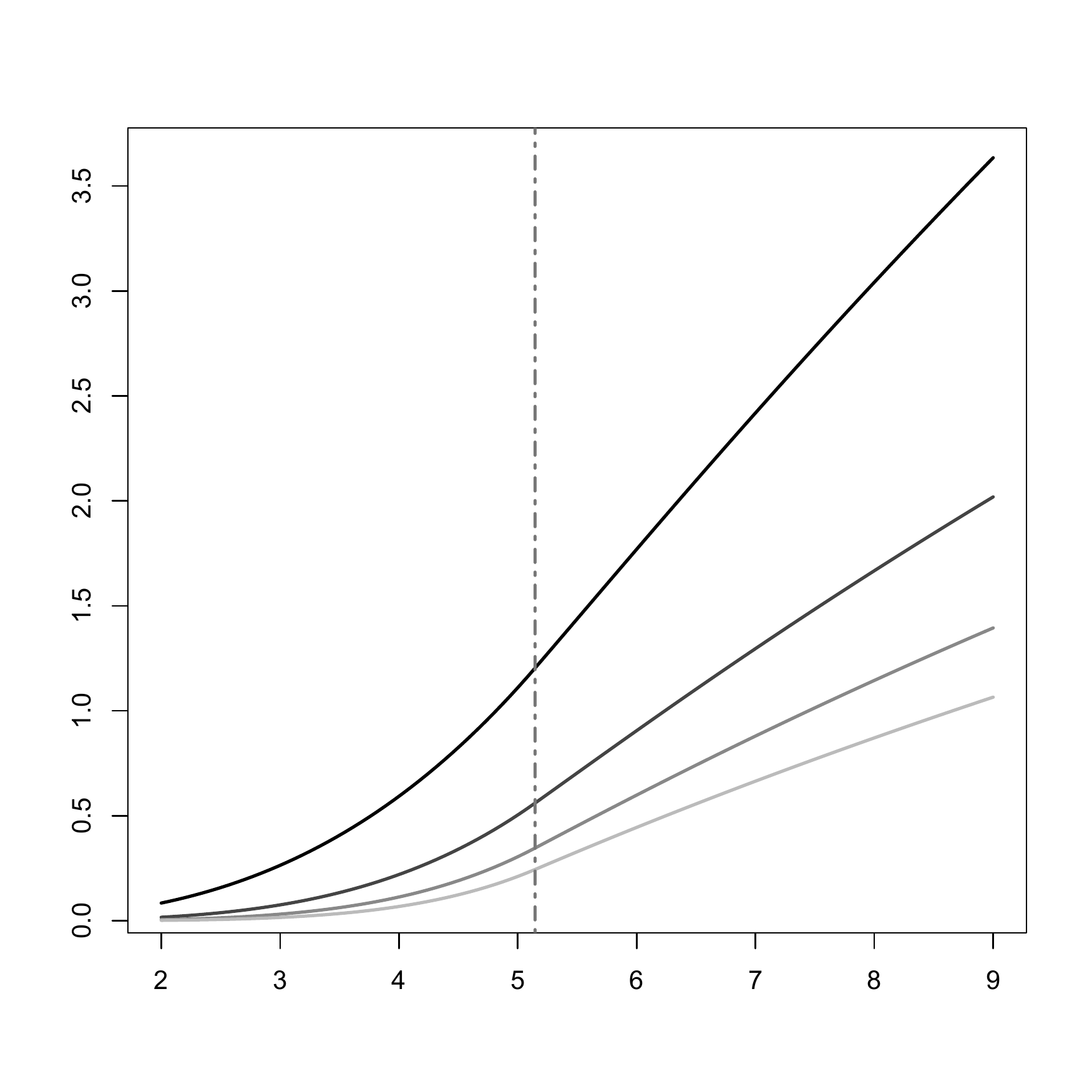}
\end{figure}

\subsection{Logistic Diffusion \cite{DP}} As a generalization of the geometric Brownian motion, we consider the diffusion $X$ with infinitesimal generator 
\[ \mathcal{A} = \frac{1}{2}\beta^2 x^2 \frac{d^2}{dx^2} + \alpha x (1-\gamma x) \frac{d}{dx}, \]
with $\alpha$, $\beta$, and $\gamma$ positive. This process exhibits mean reversion and has been applied successfully in investment theory, see \cite{DP}.  We point out that when $\gamma=0$, this process reduces to a geometric Brownian motion.

A straightforward computation yields the functions $S'$ and $m'$: $S'(x)=x^{-\frac{2\mu}{\sigma^2}}e^{\frac{2\gamma\mu}{\sigma^2}x}$ and $m'(x) =\frac{2}{(\sigma x)^2}x^{\frac{2\mu}{\sigma^2}}e^{-\frac{2\gamma\mu}{\sigma^2}x}$ for all $x\in(0,\infty)$. Furthermore, it is known from the literature that for an arbitrary $\rho>0$, the functions $\psi_\rho$ and $\varphi_\rho$ reads as
\begin{align*}
\psi_\rho(x)&=x^{b(\rho)}M(b(\rho),2b(\rho)+\frac{2\mu}{\sigma^2},\frac{2\mu\gamma}{\sigma^2}x), \\
\psi_\rho(x)&=x^{b(\rho)}U(b(\rho),2b(\rho)+\frac{2\mu}{\sigma^2},\frac{2\mu\gamma}{\sigma^2}x),
\end{align*} 
where $M$ and $U$ are, respectively, the confluent hypergeometric functions of first and second type, see \cite{DK}. Due to the complicated nature of these functions, we resort to numerical solution of the optimal entry threshold $x^*$ and the value function $G_i$.

In Figure \ref{fig3} we illustrate the curves $x \mapsto c_{i,1}\psi_{r+(1-p)\lambda}(x)$ and and $x\mapsto (R_{r+(1-p)\lambda}h_C)(x)+d_{i,2}\varphi_{r+(1-p)\lambda}(x)$ around the optimal entry threshold $x^*$ under the parameter configuration $\alpha = 0.05$, $\beta=0.25$, $r=0.1$, $K=C=\lambda=1$, $\gamma=0.2$, and $p=0.5$. The figure suggests that the curves tangent at $x^*$. 

\begin{figure}[ht]
\caption{The curves $x \mapsto c_{i,1}\psi_{r+(1-p)\lambda}(x)$ (solid black line) and $x\mapsto (R_{r+(1-p)\lambda}h_C)(x)+d_{i,2}\varphi_{r+(1-p)\lambda}(x)$ (dashes black line). The grey dashed line is located at $x^*=5.235711$. \label{fig3}}
\includegraphics[scale=0.8]{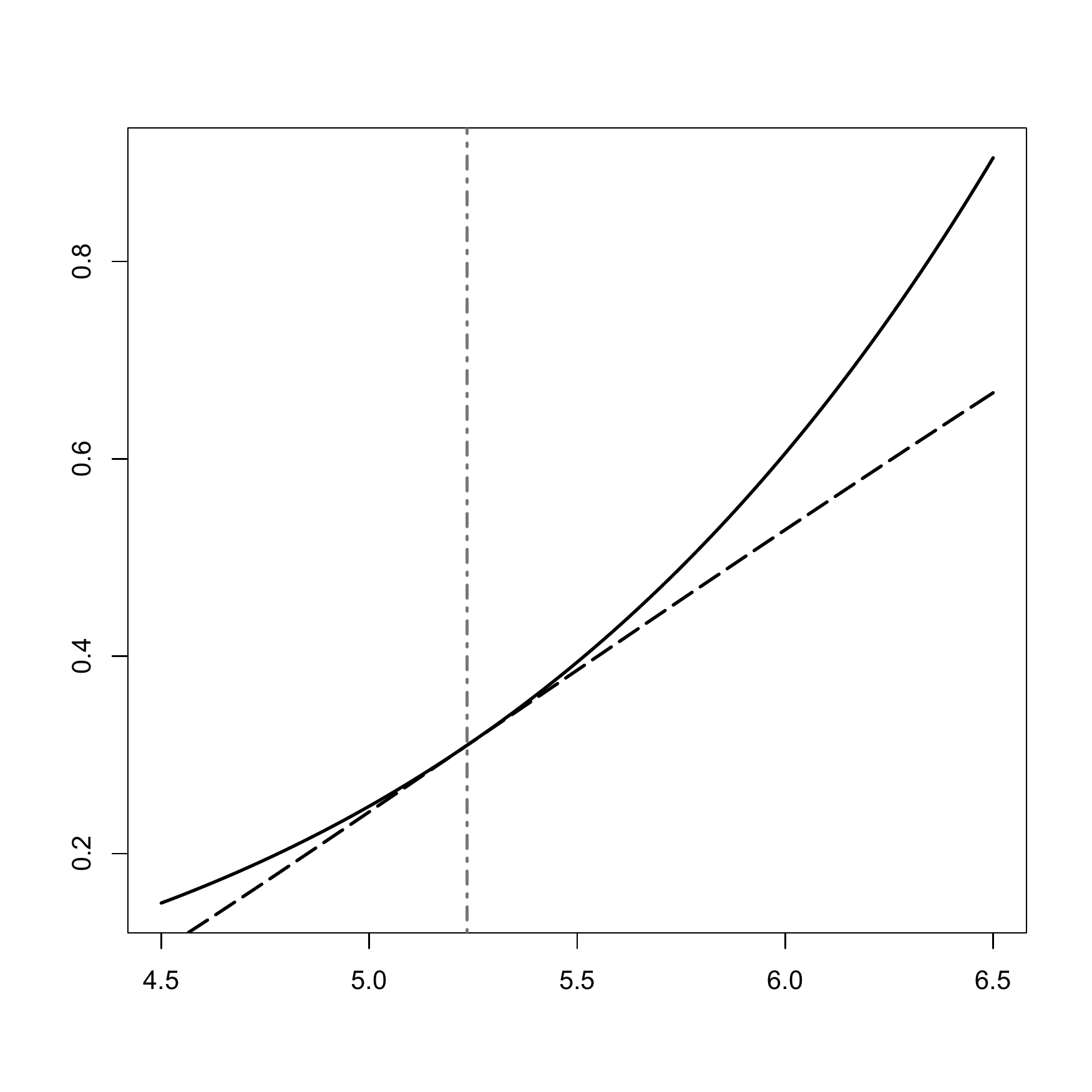}
\end{figure}

In Figure \ref{fig4} we illustrate the effect of parameter $p$ on the value $G_i$ under the parameter configuration $\alpha = 0.05$, $\beta=0.25$, $r=0.1$, $K=C=\lambda=1$, and $\gamma=0.2$. The values of $p$ are $0.8$, $0.6$, $0.4$, and $0.2$ and the curves are colored such that the hue becomes lighter as the probability $p$ decreases. The figure shows that decreasing $p$ decreases value, this is in line with our general result.
 
\begin{figure}[ht]
\caption{The value $G_i$ for different values of $p$, $p=0.8,0.6,0.4,0.2$. The hue of the curve becomes lighter as the probability $p$ decreases. The grey dashed line marks the optimal entry threshold $x^*=5.235711$. \label{fig4}}
\includegraphics[scale=0.8]{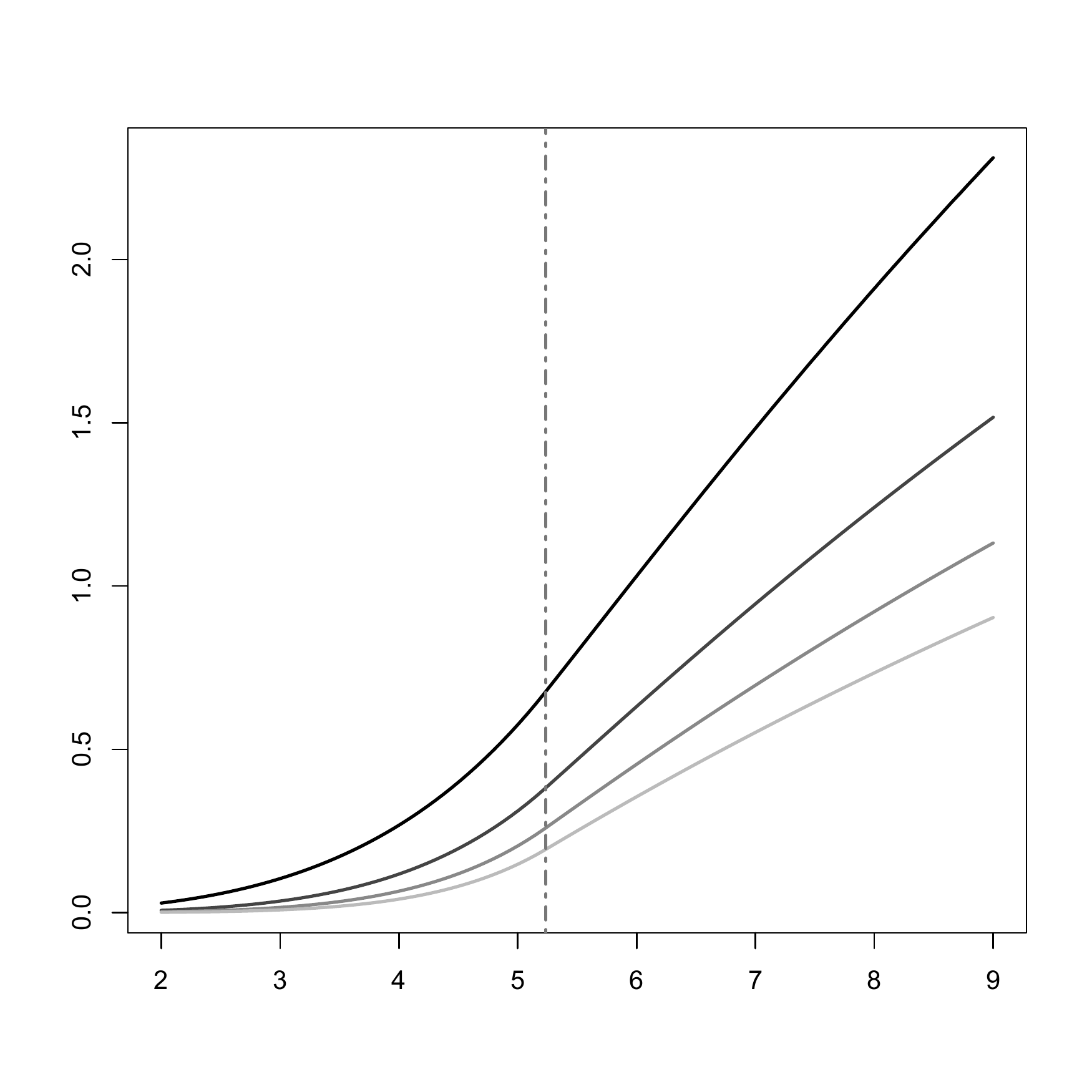}
\end{figure}

\end{document}